\documentclass[12pt,a4paper,twoside,reqno]{amsart}
\usepackage[utf8]{inputenc}
\usepackage[margin=0.9in]{geometry} 
\usepackage{amsthm}
\usepackage{amsmath}
\usepackage{amsfonts}
\usepackage{amssymb}
\usepackage{chngcntr}
\usepackage{graphicx}
\usepackage{url}
\usepackage{hyperref}
\usepackage[makeroom]{cancel}
\usepackage{enumerate}
\usepackage{mathrsfs}
\usepackage{mathtools}
\usepackage{paracol}
\usepackage{color}
\usepackage{tikz-cd}
\allowdisplaybreaks

\newtheorem{Theorem}{Theorem}

\newtheorem{Proposition}[Theorem]{Proposition}

\newtheorem{Lemma}[Theorem]{Lemma}

\newtheorem{Definition}[Theorem]{Definition}
\newtheorem{Remark}[Theorem]{Remark}
\counterwithin{Theorem}{section}
\counterwithin{equation}{section}
\emergencystretch=\maxdimen
\newcommand{\vol}{\operatorname{Vol}}
\newcommand{\spn}{\operatorname{span}}

\newcommand{\supp}{\operatorname{supp}}

\newcommand{\lex}{\operatorname{lex}}
\newcommand{\coord}{\operatorname{coord}}
\newcommand{\ex}{\operatorname{ex}}
\newcommand{\Kh}{\operatorname{Kh}}
\newcommand{\Str}{\operatorname{Str}}
\newcommand{\mat}{\operatorname{mat}}
\newcommand{\rep}{\operatorname{rep}}
\newcommand{\wt}{\operatorname{wt}}

\newcommand{\eps}{\varepsilon}

\title{Improved stability for the size and structure of sumsets}
\author{Andrew Granville, Jack Smith, and Aled Walker}
\begin{document}
\begin{abstract}
Let $A \subset \mathbb{Z}^d$ be a finite set. It is known that the sumset $NA$ has   predictable size ($\vert NA\vert = P_A(N)$ for some $P_A(X) \in \mathbb{Q}[X]$) and structure (all of the lattice points in some finite cone other than all of the lattice points in a finite collection of exceptional subcones), once $N$ is larger than some threshold. In previous work, joint with Shakan, the first and third named authors established the first effective bounds for both of these thresholds for an arbitrary set $A$. In this article we substantially improve each of these bounds, coming much closer to the corresponding lower bounds known. 
\end{abstract}
\maketitle

\section{Introduction}
Let $A \subset \mathbb{Z}^d$ be a finite set, and for each positive integer $N$ consider the sumset \[NA:= \{a_1 + \cdots + a_N: \, a_i \in A \text{ for all }  i\}.\] When $N$ is sufficiently large,  $NA$ becomes rigidly structured. In this article we study two indicators of such structure, establishing that the values of $N$ which are ``sufficiently large'' are not too large (and indeed are near to what we would guess are the smallest such $N$).

The first notion involves the size $\vert NA\vert$. Start with the convex hull of $A$, denoted \[ H(A): = \Big\{ \sum_{a \in A} c_a a: \, \text{Each } c_a \in \mathbb{R}_{ \geqslant 0}, \, \sum_{a \in A} c_a = 1\Big\}.\]  Certainly $NA \subset NH(A) \cap \mathbb{Z}^d$, and therefore $\vert NA\vert \leqslant \vert NH(A) \cap \mathbb{Z}^d\vert$.  Ehrhart showed  (\cite{E62}  \cite[Theorem 3.8]{BR15}) that there is a polynomial $R_A \in \mathbb{Q}[X]$ of degree at most $d$ for which 
\[ \vert NH(A) \cap \mathbb{Z}^d \vert = R_A(N)\] 
for all positive integers $N$. Therefore 
$\vert NA\vert \leq R_A(N)$ for all $N$ but one can readily find examples for which $\vert NA\vert < R_A(N)$ for all $N$: for instance, if $d=1$ and $A = \{0,3,5\}$ then $H(A) = [0,5]$ and   $\vert NA\vert = 5N - 5<R_A(N) = 5N + 1$ for all $N \geqslant 3$. 

Even though $\vert NA\vert$ is not equal to the Ehrhart polynomial in this example, it is still equal to a polynomial in $N$ once $N$ is sufficiently large. This was established when $A \subset \mathbb{Z}$ by Nathanson \cite{N72}, using an explicit combinatorial argument and, remarkably, this holds in arbitrary dimension: 

\begin{Theorem}[Khovanskii \cite{K92}]
\label{Theorem:Khovanskii}
Let $A \subset \mathbb{Z}^d$ be finite. There is a polynomial $P_A \in \mathbb{Q}[X]$ of degree at most $d$, and a threshold $N_{\Kh}(A)$, such that $\vert NA\vert = P_{A}(N)$ provided $N \geqslant N_{\Kh}(A)$. 
\end{Theorem}
Khovanskii's proof related the sequence $N \mapsto \vert NA\vert$ to the Hilbert function of a certain graded module over the polynomial ring $\mathbb{C}[X_1,\dots,X_{\ell}]$ (where $\ell = \vert A\vert$), and so agrees with  the Hilbert polynomial of the graded module once $N\geq N_{\Kh}(A)$. But Hilbert's proof \cite[Theorem 1.11]{E95} does not yield an explicit bound on $N_{\Kh}(A)$. Nathanson and Ruzsa \cite{NR02} later gave a combinatorial proof of Theorem \ref{Theorem:Khovanskii}, but this did not provide an effective bound on $N_{\Kh}(A)$ either, relying on the following well-known principle (proved in \cite[Lemma 5]{GS20}, say). 
\begin{Lemma}[The Mann--Dickson Lemma] 
\label{Lemma:Mann_Dickson}
For any $S \subset \mathbb{Z}^d_{ \geqslant 0}$ there exists a finite subset $S_{\min} \subset S$ such that for all $s \in S$ there exists $x \in S_{\min}$ with $s - x \in \mathbb{Z}^d_{\geqslant 0}$. 
\end{Lemma}
\noindent Prior to 2021, explicit bounds were known when $d=1$ (\cite{N72, WCC11, GS20, GW21}, with the strongest results in \cite{L22}); when $H(A)$ is a $d$-simplex (\cite{CG21}, with a refinement in \cite{GSW23}); or when $H(A)$ is $d$-dimensional and $\vert A\vert = d+1$ or $d+2$ (\cite{CG21}, with a refinement in \cite{GSW23}). 

 The first and third named authors, with Shakan in  \cite{GSW23}, proved the first effective bounds for arbitrary $d$ and arbitrary $A \subset \mathbb{Z}^d$ showing in \cite[Theorem 1.1]{GSW23}  that   
 \begin{equation}
\label{eq:first_Khov_bound}
N_{\Kh}(A) \leqslant (2 \vert A\vert \cdot \operatorname{width}(A))^{(d+4)\vert A\vert},
\end{equation}
where $\operatorname{width}(A) := \max_{a_1,a_2 \in A} \Vert a_1 - a_2\Vert_{\infty}$ is the `width' of $A$.  The proof used a complicated explicit linear algebra argument to bound $\vert S_{\min}\vert$ in the cases that the Nathanson--Ruszsa argument required.  

By returning to Khovanskii's original approach, and adapting techniques in Gr\"{o}bner bases from \cite[Chapter 4]{S96} as applied to toric ideals, we have been able to greatly improve \eqref{eq:first_Khov_bound}. To state the new bound, we define two quantities that will occur frequently throughout.

\begin{Definition}
If $A = \{a_1,\dots,a_{\ell}\} \subset \mathbb{Z}^d$, we define

\begin{align*}
\vol^{\dag,\max}(H(A)) &:= \max_{\{ i_0, i_1,\dots, i_{d}\}\subset \{1,\dots,\ell\}} \vert \det ( a_{i_1}-a_{i_0} ,  \cdots  ,  a_{i_{d}}-a_{i_0})\vert\\
\vol^{\dag,\min}(H(A)) &:= \min\limits_{\substack{\{ i_0, i_1,\dots, i_{d}\}\subset \{1,\dots,\ell\} \\   \det ( a_{i_1}-a_{i_0} ,  \cdots  ,  a_{i_{d}}-a_{i_0})  \ne 0}} \vert \det ( a_{i_1}-a_{i_0} ,  \cdots  ,  a_{i_{d}}-a_{i_0})\vert.
\end{align*}
\end{Definition}

\begin{Remark}
\emph{Note that $\frac{1}{d!}\vol^{\dag,\max}(H(A))$ is equal to the volume of the largest $d$-simplex subtended by elements of $A$. In particular $\vol^{\dag,\max}(H(A)) \leqslant d! \vol(H(A))$. By Hadamard's inequality we also have $\vol^{\dag,\max}(H(A)) \leqslant d^{d/2} \operatorname{width}(A)^d$. }
\end{Remark}

Letting $\Lambda_{A-A} \subset \mathbb{Z}^d$ denote the lattice generated by $A-A$, our first result is as follows.

\begin{Theorem}[Improved Khovanskii threshold]
\label{Theorem:improved_Khovanskii}
Suppose that $A \subset \mathbb{Z}^d$ is finite and that  $\Lambda_{A-A}$ is $d$-dimensional. Then
\begin{equation}
\label{eq:improved_Khovanskii}
N_{\Kh}(A) \leqslant \vert A\vert^2 \vol^{\dagger, \max}(H(A)) - \vert A\vert + 1.
\end{equation}
\end{Theorem}
\noindent Our proof is motivated by notions in algebraic geometry (as in \cite[Chapter 4]{S96}), but we present a simpler, more-or-less equivalent, formulation using only linear algebra. This will be useful when considering the second notion of structure for $NA$, discussed below. 

Theorem \ref{Theorem:improved_Khovanskii} implies the upper bounds\[N_{\Kh}(A) \leqslant d! \vert A\vert^2 \vol(H(A)) - \vert A\vert + 1\] and \[N_{\Kh}(A) \leqslant \vert A\vert^2 d^{d/2} \operatorname{width}(A)^d - \vert A\vert + 1,\] indicating the scale of improvement over \eqref{eq:first_Khov_bound}. These may be compared with lower bounds. For example, when $\vert A\vert = d+2$ and $\Lambda_{A-A}= \mathbb{Z}^d$ it was shown in \cite[Theorem 1.2]{CG21} that $N_{\Kh}(A) = d! \vol(H(A)) - d - 1$. This means that Theorem \ref{Theorem:improved_Khovanskii} is optimal up to the $\vert A\vert^2$ term (and the $\vert A\vert^2$ term in \eqref{eq:improved_Khovanskii} cannot be replaced by $1$, as $ \vol^{\dagger, \max}(H(A)) < d!\vol(H(A))$ for some sets with $\vert A\vert = d+2$).

The bound on $N_{\Kh}(A)$ is related to an influential conjecture in algebraic geometry called the Eisenbud--Goto regularity conjecture \cite{MP18}. Though now known to be false in full generality, the conjecture may still be true for projective toric varieties, which is the relevant case for bounding $N_{\Kh}(A)$. A proof of this case of the conjecture would imply $N_{\Kh}(A) \leqslant d! \vol(H(A)) - \vert A\vert + O_d(1)$ which, given the above comments on the bounds when $\vert A\vert = d + 2$,  would be essentially optimal. We direct the interested reader to \cite[Conjectures 4.1 and 4.2]{S96a}, also available at \cite{S96b}, and to \cite[Chapter 4]{S96}. 

For the second notion of structure, we consider   the inclusion $NA \subset NH(A) \cap \mathbb{Z}^d$ in more detail, an inclusion introduced in \cite{GS20}, with antecedents in Khovanskii's original paper \cite{K92}. Let $\ex(H(A))\subset A$ denote the set of extremal points of the polytope $H(A)$, and translate $A$  so that $0 \in \ex(H(A))$ and $\Lambda_{A} = \mathbb{Z}^d$, without loss of generality. Then $NH(A) \subset C_A$ where 
\[C_A: = \Big\{ \sum_{a \in A} c_a a: \, c_a \in \mathbb{R}_{ \geqslant 0} \text{ for all } a\Big\}\] 
is the cone generated by $A$, and the semigroup generated by $A$ is the nested union (as $0 \in A$)
 \[ \mathcal{P}(A): = \bigcup\limits_{N=1}^{\infty} NA  \subset C_A \cap \mathbb{Z}^d.\] 
 The set of \emph{exceptional elements} are those lattice points in $C_A$ which do not belong to $\mathcal{P}(A)$,
\[\mathcal{E}(A):= (C_A \cap \mathbb{Z}^d) \setminus \mathcal{P}(A),\]
and so
 \[NA \subset (NH(A) \cap \mathbb{Z}^d) \setminus \mathcal{E}(A).\]
  Similarly, for all $a \in \ex(H(A))$ we have $0 \in a - \ex(H(A)) = \ex(H(a - A))$. Since $\Lambda_{a - A} = \mathbb{Z}^d$ too, we have \[ N(a-A) \subset (N H(a-A) \cap \mathbb{Z}^d) \setminus \mathcal{E}(a-A).\] Rearranging and taking the intersection over all $a \in \ex(H(A))$, we get 
\begin{equation}
\label{eq:structure_inclusion}
NA \subset (NH(A) \cap \mathbb{Z}^d) \setminus \Big( \bigcup\limits_{a \in \ex(H(A))} (aN - \mathcal{E}(a-A))\Big).
\end{equation}

It was shown in \cite{GS20}  that there is equality in \eqref{eq:structure_inclusion} once $N \geqslant N_{\Str}(A)$; that is 
\begin{equation}
\label{eq:structure_equation}
NA = (NH(A) \cap \mathbb{Z}^d) \setminus \Big( \bigcup\limits_{a \in \ex(H(A))} (aN - \mathcal{E}(a-A))\Big),
\end{equation}
filling out to  its maximal possible size. This was proved by Nathanson \cite{N72} when $d=1$  and for $d\geqslant 2$  in \cite{GS20}; however the proof in \cite{GS20} did not produce a value for $N_{\Str}(A)$ as it relied on the ineffective Lemma \ref{Lemma:Mann_Dickson}.
The article \cite[Theorem 1.3]{GSW23} then gave the first effective bound on $N_{\Str}(A)$ for all $A$:
 \begin{equation}
\label{eq:weak_structure_bound}
N_{\Str} \leqslant (d \vert A\vert \cdot \operatorname{width}(A))^{13 d^6}.
\end{equation}
Previous bounds were known when $d=1$ \cite{N72, WCC11, GS20, GW21, L22} and when $H(A)$ is a $d$-simplex (\cite{CG21}, with refined bounds in \cite{GSW23}). 

The proof of \eqref{eq:weak_structure_bound} in \cite{GSW23} was intricate involving an ``induction on dimension'' strategy. This required  repeated use of Siegel's Lemma from quantitative linear algebra (in the version proved by Bombieri--Vaaler \cite{BV83}) together with delicate geometric considerations, such as the size and shape of the intersection between neighbourhoods of two cones $C_A$ and $C_B$. 

Our second main result  gives a strengthening of \eqref{eq:weak_structure_bound}, with a much simpler proof. This is based in part on ideas from the proof of Theorem \ref{Theorem:improved_Khovanskii}, developed out of the ideas in \cite[Chapter 4]{S96}. Before stating the result, we introduce one final quantity associated to $A$. 

\begin{Definition}
\label{def: kappa}
Let $A \subset \mathbb{R}^d$ be finite with $\spn(A-A) = \mathbb{R}^d$. Given a facet (i.e.~$(d-1)$-dimensional face) $F$ of $H(A)$, and a point $a \in A \setminus F$, let $\vol(F, a)$ denote the volume of the polytope given by the convex hull of $F$ and $a$.  Then set
\[
\kappa(A) = \max_F \frac{\max_a \vol(F, a)}{\min_a \vol(F, a)}.
\]
\end{Definition}

\begin{Remark}
\label{rmk: kappa}
\emph{There are several equivalent ways to define $\kappa(A)$. Indeed, for each $F$ we could equivalently replace $\vol(F, a)$ by $g_F(a)$ for any affine-linear function $g_F: \mathbb{R}^d \to \mathbb{R}$ which vanishes on $F$ and is strictly positive on $H(A) \setminus F$.  For example, we could take $g_F(a)$ to be the (signed) orthogonal distance from $F$ to $a$.  Or we could pick linearly independent points $b^{(1)}, \dots, b^{(d)}$ in $F \cap A$ and let $g_F(a) = \det(b^{(1)} - a, \dots, b^{(d)} - a)$, where the $b^{(j)}$ are ordered to make this determinant positive for $a \in H(A) \setminus F$.  Using the latter choice of $g_F$ we see that}
\begin{equation}
\label{eq: kappa vol bound}
\kappa(A) \leqslant \frac{\vol^{\dag,\max}(H(A))}{\vol^{\dag,\min}(H(A))}.
\end{equation}
\end{Remark}

Our main result  is  the first  general effective bound for $N_{\Str}(A)$ that captures the geometry of $A$ by involving the quantities $\kappa(A)$, $\vol(H(A))$, and $\vol^{\dag,\max}(H(A))$:

\begin{Theorem}[Improved structural threshold]
\label{Theorem:improved_structural_threshold}
Let $A \subset \mathbb{Z}^d$ be a finite set, with $0 \in \ex(H(A))$ and $\Lambda_A = \mathbb{Z}^d$. Then we have the following two upper bounds:
\begin{equation}
\label{eq:first_structure_bound}
N_{\Str}(A) \leqslant (d+1)\kappa(A) \Big(d!\vol(H(A)) + \big(\lvert\ex(H(A))\rvert - d - 1\big) \vol^{\dag,\max}(H(A))\Big)
\end{equation}
and
\begin{equation}
\label{eq:second_structure_bound}
N_{\Str}(A) \leqslant  (d+1)\kappa(A) \big(\vert A\vert - d - 1\big) \vol^{\dag,\max}(H(A)).
\end{equation}
\end{Theorem}
\noindent The   bound \eqref{eq:first_structure_bound} is better when $\vert A\vert$ is substantially larger than $\ex(H(A))$; and \eqref{eq:second_structure_bound} when $d! \vol(H(A))$ is substantially larger than $\vol^{\dag,\max}(H(A))$. Using \eqref{eq: kappa vol bound} and bounding $\vol^{\dag,\min}(H(A)) \geqslant 1$ and $\vol^{\dag,\max}(H(A)) \leqslant d! \vol(H(A))$, \eqref{eq:first_structure_bound} implies the cleaner but slightly weaker bound
\begin{equation}
\label{eq:clean_structure_bound}
N_{\Str}(A) \leqslant (d+1)(d!)^2 (\vert \ex(H(A))\vert - d) \vol(H(A))^2,
\end{equation}
but still much stronger than \eqref{eq:weak_structure_bound}. 

Similar to the Khovanskii threshold, we guess that a bound like $N_{\Str}(A) \leqslant d! \vol(H(A))$ holds in general. Our \eqref{eq:clean_structure_bound} is roughly the square of this bound, so still far from  optimal.  

If $H(A)$ is a simplex then $\kappa(A) = 1$ and $\lvert \ex(H(A))\rvert = d+1$, so \eqref{eq:first_structure_bound} implies that $N_{\Str}(A) \leqslant (d+1)! \vol(H(A))$, which essentially recovers the best known bound in this case \cite[Theorem 1.5]{GSW23}.  

The paper is structured as follows. In Section \ref{Section:Khovanskii} we prove Theorem \ref{Theorem:improved_Khovanskii}, while developing   general lemmas on equations $\sum_{a \in A} a = \sum_{b \in B} b$ (with $B \subset A$) that will be useful throughout. In Section \ref{Section:structure} we use these lemmas,  with some convex geometry, to deduce Theorem \ref{Theorem:improved_structural_threshold}.

\subsection*{Acknowledgements}
This material is partly based upon work supported by the Swedish Research Council under grant no. 2021-06594 while the third author was in residence at Institut Mittag-Leffler in Djursholm, Sweden, during the winter semester of 2024. Our proof of Theorem \ref{Theorem:improved_Khovanskii} is pretty much that presented in \cite[Chapter 4]{S96}, albeit written in a different mathematical language and context; it also helped inspire the proof of \ref{Theorem:improved_structural_threshold}. \\

\section{Proof of Theorem \ref{Theorem:improved_Khovanskii}}
\label{Section:Khovanskii}

Let $A = \{a_1,\dots,a_{\ell}\}$. To keep track of various quantities throughout the proof, we define the \emph{weight} of a vector $m\in \mathbb Z^\ell$ by $\wt(m):=m\cdot 1=\sum_{i=1}^\ell m_i$. We also let $A_{\mat}:=(a_1,\dots,a_\ell)$ be the $d$-by-$\ell$ matrix formed with the $a_i$ as column vectors, so that  $A_{\mat} m=\sum_i m_ia_i$. 

We begin with a result and proof due to Nathanson and Ruzsa. 

\begin{Proposition} \label{Nath-Ruz} There exists a finite set of lattice points $\mathcal M\subset  \mathbb Z_{\geqslant 0}^\ell$, as described in the proof, such that for all positive integers $h$ we have  
\[
|h A|= \sum_{ T\subset \mathcal M } (-1)^{|T|}  \binom{h-\wt(m_T)+\ell-1}{\ell-1},
\]
where $m_T$ is the vector with $ (m_T)_i:=\max_{m\in T} (m)_i$, and $\binom N{\ell-1}=0$ if $N<\ell-1$.
\end{Proposition}
\begin{proof}  
If $x\in h A$ then let 
\[
\rep_h(x):=\{ m\in \mathbb Z_{\geqslant 0}^\ell: \wt(m)=h \text{ and } A_{\mat} m =x\}, 
\] denote the coefficient set of non-negative combinations of $h$ elements of $A$ that represent $x$. Let $m_h(x)$ be the minimum element in $\rep_h(x)$ with respect to the lexicographic ordering, and let 
\[
\mathcal U:= \bigcup_{h\geqslant 0} \bigcup_{x\in h A} \{ m\in \rep_h(x): m\ne m_h(x)\}.
\]
Evidently $\vert h A\vert = \vert \{ m_h(x): x\in hA\}\vert$. We will calculate this size using the relationship 
\[
\{ m_h(x): x\in h A\}= \{n \in \mathbb Z_{\geqslant 0}^\ell:\ \wt(n)=h\text{ and }   n\not\in \mathcal U \} .
\]

To this end, note that $\mathcal U+\mathbb Z_{\geqslant 0}^\ell=\mathcal U$. Indeed, if 
$y\in \mathbb Z_{\geqslant 0}^\ell$ then writing $v=A_{\mat} y $ and $k=\wt(y)$ we have
$\rep_h(x) +y \subset \rep_{h+k}(x+v)$. So if $m\in \rep_h(x)\cap \mathcal{U}$ then
\[
m+y >_{\lex}  m_h(x)+y \geq_{\lex} m_{h+k}(x+v).
\] So $m+y \in \mathcal{U}$ as needed. 

We now apply Lemma \ref{Lemma:Mann_Dickson} (the Mann--Dickson Lemma) to $\mathcal U$. Writing $x\leq_{\coord} y$ if $x_i\leqslant y_i$ for all $i$, this implies the set
\[\mathcal{M} = 
\mathcal M(\mathcal U):=\{ m\in \mathcal{U}:   \text{for all }u \in \mathcal{U}, \, u\leq_{\coord} m\implies u=m\} 
\]
 of minimal elements is finite, and for every $u \in \mathcal{U}$ there exists some $m \in \mathcal{M}$ with $m \leq_{\coord} u$. Therefore we can use inclusion-exclusion to obtain
\begin{align*}
& \{n \in \mathbb Z_{\geqslant 0}^\ell:\ \wt(n)=h\text{ and }   n\not\in \mathcal U \}  \\
&\qquad
=  \sum_{ T\subset \mathcal M} (-1)^{|T|} \{ n \in \mathbb Z_{\geqslant 0}^\ell:\ \wt(n)=h\text{ and } m\leq_{\coord} n\ \forall m\in T\}\\
&\qquad =  \sum_{ T\subset \mathcal M} (-1)^{|T|} \{ n \in \mathbb Z_{\geqslant 0}^\ell:\ \wt(n)=h\text{ and } m_T \leq_{\coord} n \}.
\end{align*}
 We have written this in terms of sets, and one should think of $+S$ as including the elements of $S$ with multiplicity, and $-S$ as removing one copy of each element of $S$. 

Note that if $n \in \mathbb{Z}_{ \geqslant 0}^\ell$ with $\wt(n) = h$ and $m_T \leq_{\coord} n$, then $\wt(m_T)\leqslant h$. In that case, writing $n=m_T+r$ we obtain
 \[
  \{ n \in \mathbb Z_{\geqslant 0}^\ell:\ \wt(n)=h\text{ and } m_T\leq_{\coord} n\}
  = \{ r \in \mathbb Z_{\geqslant 0}^\ell:\ \wt(r)=h-\wt(m_T)\} 
   \]
   and the result follows from the usual `stars and bars' bound. 
\end{proof}
The remainder of the proof of Theorem \ref{Theorem:improved_Khovanskii} concerns bounding $\wt(m_T)$ above. To describe this argument we introduce some more notation, which will be of use throughout the paper. Given  $u  \in \mathbb{Z} ^{\ell}$ define vectors $u^+,u^- \in \mathbb{Z}_{\geqslant 0}^{\ell}$ where 
$(u^+)_i=\max\{ 0, u_i\}$ and $(u^-)_i=\max\{ 0, -u_i\}$, so that $u = u^+ - u^{-}$. If $u \in \mathbb{Z}^{\ell}$, we let $\supp(u) = \{i: \, u_i \neq 0\}$. 
Next let  
\begin{equation}
\label{eq:def_of_Z}
\mathcal Z = \mathcal{Z}(A):= \{ z\in \mathbb Z^\ell:  \wt(z)=0 \text{ and } A _{\mat}z=0\} 
\end{equation}
which is a lattice. By taking $m=z^+, m'=z^-$ with $x=A_{\mat} m$ and $h=\wt(m)$ we obtain
\[
\mathcal Z=\bigcup_{h\geqslant 0} \bigcup_{x\in h A} \{ m-m': m,m'\in \rep_h(x)\}.
\]

We continue with a lemma relating $\mathcal{M}$ and $\mathcal{Z}$. 

\begin{Lemma}
\label{Lemma minimal useless elements}
Given $m \in \mathcal{M}$ let $x=A_{\mat} m$ and $h=\wt(m)$.  Then
 $\supp(m) \cap \supp(m_h(x)) = \emptyset$.  Moreover if there exists 
$v \in \mathcal{Z} \setminus \{0\}$  with $v^+\leq_{\coord} m$ and $v^-\leq_{\coord} m_h(x)$
then $v^+ = m$, $v^- = m_h(x)$, and $v=m-m_h(x)$.
\end{Lemma}

\begin{proof} Write $n=m_h(x)$ so that $A_{\mat}(m-n)=x-x=0$.
If $\supp(m) \cap \supp(n)\neq \emptyset$, say that $m_i,n_i\geqslant 1$. Then
$m-e_i>_{\lex} n-e_i$ with $m-e_i, n-e_i\in \mathbb Z_{\geqslant 0}^\ell$ so $m-e_i\in \mathcal U$, and $m-e_i<_{\coord} m$, contradicting that $m\in \mathcal M(\mathcal U)$. So $\supp(m) \cap \supp(n)=\emptyset$ as claimed. 

Now if $v\in \mathcal{Z} \setminus \{0\}$  with $v^+\leq_{\coord} m$ and $v^-\leq_{\coord} n$ then $v^+>_{\lex} v^-$. Indeed, if not then  
    $ w:= v^++ n - v^-  <_{\lex}  n$ and $w \in \mathbb{Z}_{\geqslant 0}^{\ell}$. Moreover 
    $m - w= (m-n) - (v^+- v^-)  \in  \mathcal Z$, 
    so that $w\in \rep_h(x)$ with $w<_{\lex} n$. However $n= m_h(x)\leq_{\lex} w$ by definition, which gives a contradiction. So $v^+ >_{\lex} v^-$. 

Finally, let $y=A_{\mat} v^+ $ and $k=\wt(v^+)$, so that $v^+,v^- \in \rep_k(y)$ and  $v^+\in \mathcal{U}$ (as $v^+ >_{\lex} v^-$). Then $v^+\leq_{\coord}  m$ and  $m \in \mathcal{M}$, so $v^+=m$. Therefore $y=x, k=h$ and so $v^-\in \rep_h(x)$ which implies that 
$n\leq_{\lex} v^-$. Moreover $v^-\leq_{\coord} n$ which implies $v^-\leq_{\lex} n$, and so $v^-=n$. So $v = v^+ - v^- = m-n$ as claimed. 
\end{proof}

We define $\mathcal{Z}^\dag = \mathcal{Z}^\dag(A)$ by
\[\mathcal{Z}^\dag:= \{ u\in \mathcal Z \setminus \{0\}:   \text{ If } v \in \mathcal Z \setminus \{0\} \text{ with } 
\supp(v)\subset  \supp(u) \text{ then } v=\lambda u \text{ for some } \lambda \in \mathbb Z\}.
\]
Note that if $v\in \mathcal Z \setminus \{0\}$ then there must exist some $u\in \mathcal Z^\dag$ with 
$\supp(u)\subset  \supp(v)$. It transpires that elements in $\mathcal{Z}^\dag$ may be strongly controlled, and this in turn will help control $\mathcal{Z}$ and finally $\mathcal{M}$. 

\begin{Lemma}
\label{Lemma circuit height bound}
If $u \in \mathcal Z^\dag$ then $\Vert u\Vert_{\infty}:=\max_i |u_i|   \leqslant \emph{Vol}^{\dag,\max}(H(A))$. 
\end{Lemma}
\begin{proof} For each $a_i\in  A$ let $b_i=({a_i\atop 1}) \in \mathbb{Z}^{d+1}$ \  If $\supp(u) = \{i_1,\dots,i_r\}$ then the only 
 linear dependence (up to scalars) amongst the the vectors $\{  b_{i_1},\cdots, b_{i_r}\}$ is 
$\sum_j u_{i_j}b_{i_j}=0$, so that  the  $(d+1)$-by-$r$ matrix $M = (b_{i_1},\cdots, b_{i_r})$ has rank $r-1$. 
Since $\spn(\{b_1,\dots,b_{\ell}\}) = \mathbb{R}^{d+1}$ (as $\Lambda_{A-A}$ is $d$-dimensional), we can find column vectors $ b_{i_{r+1}},\dots b_{i_{d+2}}$ such that the  $(d+1)$-by-$(d+2)$ matrix $M^\prime =  (b_{i_1},\cdots, b_{i_{d+2}})$  has rank $d+1$ and so has a 1-dimensional null space.  Cramer's rule gives a non-zero null vector
\begin{equation}
\label{key vector}
w:= \sum\limits_{j=1}^{d+2} (-1)^j \det(b_{i_1}\cdots ,b_{i_{j-1}},b_{i_{j+1}},\cdots b_{i_{d+2}})  \cdot  e_{i_j}
 \end{equation}
(non-zero as the subdeterminants cannot all be zero  since $M^\prime$  has rank $d+1$). We already have the null vector $u$, so $w$ and $u$ must be scalar multiples of one another. In particular $\supp(w) \subset \supp(u)$, and hence $w = \lambda u$ for some $\lambda \in \mathbb{Z}$. This implies that
$|u_{i_j}|\leqslant |w_{i_j}|\leqslant |\det(b_{i_1}\cdots ,b_{i_{j-1}},b_{i_{j+1}},\cdots b_{i_{d+2}}) |$ for all $j$, and hence
\[
\Vert u\Vert_{\infty} \leqslant \max_{\{ k_0, k_1,\dots, k_{d}\}\subset \{1,\dots,\ell\}} \bigg\vert \det \bigg( \bigg( { a_{k_0}\atop 1}\bigg) ,  \cdots  , \bigg( { a_{k_d}\atop 1}\bigg) \bigg) \bigg\vert = \text{Vol}^{\dag, \max}(H(A)).
\]
  For the final equality, we have relabelled  $\{ b_{i_1}, \cdots , b_{i_{j-1}}, \ b_{i_{j+1}},\cdots  b_{i_{d+2}}\}$ as $\{ a_{k_0},\dots, a_{k_d}\}$ and then subtracted the first column from the others, expanding the determinant about the bottom row. 
\end{proof}
We continue by relating $\mathcal{Z}^{\dag}$ and $\mathcal{Z}$. To this end we write $\supp(u^\pm)\subset  \supp(v^\pm)$ as shorthand for the two conditions $\supp(u^+)\subset  \supp(v^+)$ and $\supp(u^-)\subset  \supp(v^-)$

\begin{Lemma}
\label{Lemma circuits exist}
If $v\in \mathcal Z\setminus \{0\}$ then there exists  $u\in \mathcal Z^\dag$ such that  
$\supp(u^\pm)\subset  \supp(v^\pm)$.
\end{Lemma}
\begin{proof}  Let $w=v/\gcd_i v_i$.  If $w \in  \mathcal Z^\dag$ let 
$u=w$ and we are done, so we may assume that $w \in  \mathcal Z \setminus \mathcal Z^\dag$. If $\vert \supp(v)\vert = 1$ then $w \in \mathcal Z^\dag$ automatically, so we assume that $\vert \supp(v)\vert \geqslant 2$ and proceed by induction on $\vert \supp(v)\vert$. Since $w \in \mathcal{Z} \setminus \mathcal{Z}^\dag$ is non-zero there exists $u  \in \mathcal Z \setminus \{0\}$
with  $\supp(u)\subset  \supp(w)$  but which is not an integer multiple of  $w$.
Select $\lambda:=\min_{i: u_i\ne 0}  |w_i/u_i|$, which is $>0$ as $\supp(u)\subset  \supp(w)$.
Pick $i$ so that  $w_i=\pm \lambda u_i$ with $u_i\ne 0$, and then 
adjust the sign of $u$ so that  $u_i>0$. Now let $y:=u_i w-w_i u$, so that $y_i=0$ and for all $j$ either 
$y_j$ equals $0$ or has the same sign as $w_j$, since 
$|w_i  u_j|=\lambda |u_i  u_j| = |u_i| \cdot \lambda | u_j|\leqslant |u_i  w_j|$.
Therefore $\supp(y)\subset  \supp(w)\setminus\{ i\}$ with 
$\supp(y^\pm)\subset  \supp(w^\pm)$. Note further that $y \neq 0$, since if $0=y = u_i w - w_i u$ then $w_i \neq 1$, since $u$ is not an integer multiple of $w$, but this in turn contradicts the coprimality of the coordinates of $w$. 

Since $y \in \mathcal{Z} \setminus \{0\}$, by the induction hypothesis there exists 
$u\in  \mathcal Z^\dag$ for which  
\[
\supp(u^\pm)\subset \supp(y^\pm)\subset \supp(w^\pm)=\supp(v^\pm). \qedhere
\]
\end{proof}

We may iterate this argument to entirely decompose elements $v \in\mathcal{Z}$ in terms of a combination of elements $u \in \mathcal{Z}^\dag$. 

\begin{Lemma}[Decomposing using $\mathcal{Z}^\dag$]
\label{Lemma rational combo}
Any $v  \in  \mathcal Z \setminus \{0\}$ can be written as $\sum_{j=1}^I \lambda_j u_j$ with each $u_j\in  \mathcal Z^\dag, \lambda_j\in \mathbb Q_{>0}$  and $I\leqslant \vert\supp(v)\vert$. Furthermore each  $\supp(u_j^\pm)\subset  \supp(v^\pm)$, so that $v^+= \sum_{j=1}^I \lambda_j u_j^+$ and $v^-= \sum_{j=1}^I \lambda_j u_j^-$.
\end{Lemma}

\begin{proof}  
By induction on $m:=\vert \supp(v)\vert$.
Select $u\in \mathcal Z^\dag$ by Lemma \ref{Lemma circuits exist},    let
$\lambda:=\min_{i: u_i\ne 0}  v_i/u_i$ (noting $v_i$ and $u_i$ have the same sign as $\supp(u^\pm)\subset  \supp(v^\pm)$), choose $i$ so that $v_i=\lambda u_i$, and let
$y:=  v- \lambda u$. Now 
$\supp(y)\subset  \supp(w)-\{ i\}$ so that $|\supp(y)|\leqslant m-1$. If $y=0$ (for example if $m=1$) then $v=\lambda u$. Otherwise the result follows by the induction hypothesis. 
\end{proof}

 The preceding lemmas may be combined to control the size of elements in $\mathcal{M}$.

\begin{Lemma}
\label{Theorem key bound on primitive elements}
If $m\in \mathcal{M}$ then $\Vert m \Vert_{\infty}\leqslant \ell \, \emph{Vol}^{\dag,\max}(H( A))$.  
\end{Lemma}

\begin{proof}  Let $x=A_{\mat} m$ and $h=\wt(m)$. By Lemma \ref{Lemma minimal useless elements},
letting $u:=m-m_h(x)\in \mathcal{Z}$ we have $u \neq 0$, and if  $v \in \mathcal{Z} \setminus \{0\}$ with $v^+ \leq_{\coord} u^+$ and $v^- \leq_{\coord} u^-$ then $v=u$. Now by Lemma  \ref{Lemma rational combo} we can write
$u=\sum_{j=1}^I \lambda_j u_j$ with each $u_j\in  \mathcal Z^\dag, \lambda_j\in \mathbb Q_{>0}$  and $I\leqslant \vert \supp(u)\vert$, where each  $\supp(u_j^\pm)\subset  \supp(u^\pm)$, so that $u^+ = \sum_{j=1}^I \lambda_j u_j^+$ and $u^- = \sum_{j=1}^I \lambda_j u_j^-$.

We claim each $\lambda_j\leqslant 1$, else
  $u_j^+ <_{\coord} \lambda_j u_j^+ \leq_{\coord} u^+$ and $u_j^- <_{\coord} \lambda_j u_j^- \leq_{\coord}  u^-$. Then applying Lemma \ref{Lemma minimal useless elements} as above with $v=u_j$ we conclude $u_j = u$, but this contradicts the strict inequality  $u_j^+ <_{\coord} u^+$. 

 Therefore, by the triangle inequality,
\[
\Vert m\Vert_{\infty}  \leqslant \Vert u\Vert_{\infty}  \leqslant \sum_{j=1}^I \lambda_j \Vert u_j\Vert_{\infty} \leqslant I \cdot \max_j \Vert u_j\Vert_{\infty} 
\leqslant \ell \cdot \max_{u\in  \mathcal Z^\dag} \Vert u\Vert_{\infty} 
\]
and the   result then follows from Lemma \ref{Lemma circuit height bound}.
\end{proof}

Substituting this control on $\mathcal{M}$ into the Proposition \ref{Nath-Ruz} will quickly resolve Theorem \ref{Theorem:improved_Khovanskii}. 
\begin{proof} [Proof of Theorem \ref{Theorem:improved_Khovanskii}]
Define
\[
P_ A(x) :=  \frac{ 1}{(\ell-1)!} \sum_{ T\subset \mathcal M } (-1)^{|T|}   (x-\wt(m_T)+\ell-1)\cdots (x-\wt(m_T)+1).
\]
We observe that 
\[
  \frac{(h-\wt(m_T)+\ell-1)\cdots (x-\wt(m_T)+1)}{(\ell-1)!} =  \binom{h-\wt(m_T)+\ell-1}{\ell-1}
\]
for all integers $h\geqslant \wt(m_T)-\ell+1$. Therefore, by Proposition \ref{Nath-Ruz},
\[
h\mathcal A=P_ A(h) \text{ for all } h\geqslant \wt(m_{\mathcal M(\mathcal U)})-\ell+1
\]
since $\max_{T\subset \mathcal M(\mathcal U)}  w_T= \wt(m_{\mathcal M(\mathcal U)})$, by definition.
Hence
 \[
N_{\operatorname{Kh}}( A)+\ell-1 \leqslant  \wt(m_{\mathcal M(\mathcal U)}) = \sum_{i=1}^\ell  \max_{m\in \mathcal M(\mathcal U)} |m|_i  \leqslant \ell \, \max_{m\in \mathcal M(\mathcal U)} \Vert m\Vert_\infty \leqslant \ell^2 \, \text{Vol}^{\dag,\max}(H(A)).
\]
 by Lemma \ref{Theorem key bound on primitive elements}.  This is the claimed bound on $N_{\Kh}(A)$. 
 \end{proof}

 
\section{Proof of Theorem \ref{Theorem:improved_structural_threshold}}
\label{Section:structure}

Let us first describe the general strategy. Let $A \subset \mathbb{Z}^d$ be finite with $0 \in \ex(H(A))$ and $\Lambda_A = \mathbb{Z}^d$. If $v \in \mathcal{P}(A)$, we aim to find $u$ and $w$ such that $v = u+w$, where:
\begin{itemize}
\item $u \in MA$, for a bounded $M$;
\item $w \in \mathcal{P}(B \cup \{0\})$, where $B \subset A$ is contained within a single facet of $H(A)$. One may also assume that this facet does not contain the origin. 
\end{itemize}
In some ways, this strategy is similar to \cite[Lemma 7.1]{GSW23}. However, in \cite[Lemma 7.1]{GSW23} the set $B$ was pre-determined at the outset, with the further assumptions that $v \in \mathcal{P}(A) \cap C_B$ and the further requirement that $u \in C_B$. It turns out to be much easier to prove the weaker version outlined above, where $B$ is found as a consequence of the decomposition $v = u+w$ rather than being fixed in the hypotheses. 

To find the decomposition $v = u+w$,  one may consider a representation $v = \sum_{i=1}^\ell \eta_i a_i$ in which $\eta \in \mathbb{Z}_{ \geqslant 0}^\ell$ and the weight $\wt(\eta)$ is minimal.  Recall that $A = \{a_1, \dots, a_\ell\}$ and $\wt(\eta)$ denotes $\sum_i \eta_i$.  In this section it will actually be more convenient to index $\eta$ directly by $A$ itself, so $v = \sum_{a \in A} \eta_a a$ and $\wt(\eta) = \sum_{a \in A} \eta_a$.  The basic idea is then to let
\[
u = \sum_{\substack{a \in A \\ \eta_a \text{ is small}}} \eta_a a \quad \text{and} \quad w = \sum_{\substack{a \in A \\ \eta_a \text{ is large}}} \eta_a a.
\]
If the set $\{a \in A: \, \eta_a \text{ is large}\}$ is not contained within a single facet of $H(A)$, one can use properties of the sets $\mathcal{Z}(A)$ and $\mathcal{Z}^\dag(A)$ established previously (Lemmas \ref{Lemma circuit height bound} and \ref{Lemma rational combo}) to reduce $\wt(\eta)$, contradicting minimality. 

Having proved this decomposition, suppose $v \in NH(A)$ as well, with $N$ at least the right-hand side of \eqref{eq:second_structure_bound}.  If the facet on which $B$ lies is defined by $\beta = 1$ for a linear map $\beta:\mathbb{R}^d \to \mathbb{R}$, then one can apply $\beta$ to both sides of the equation $v = u+w$. Writing $w = \sum_{b \in B} \lambda_b b$, we have $\beta(w) = \wt(\lambda)$, and this enables us to bound $\wt(\lambda)$ above in terms of $N$ and $M$. Putting everything together, we can place $v \in NA$ as required. (We extend the definition of $\wt$ to mean simply the sum of the entries of a vector.  We will also implicitly allow ourselves to enlarge the indexing set of a vector, by setting all previously undefined entries to zero.)

This method gives \eqref{eq:second_structure_bound}. In order to prove \eqref{eq:first_structure_bound}, which involves $\lvert\ex(H(A))\rvert$ instead of $\lvert A\rvert$, one first excises the contribution from non-extremal elements (Lemma \ref{Lemma:minelement_decomp} below). This is a simple additive--combinatorial argument, adapted from similar results in \cite{CG21} and \cite{GSW23}. This done, one proceeds as above but with $A$ replaced by $\ex(H(A))$. \\

To begin the proof proper, we state some standard results on convex polytopes. Let $A \subset \mathbb{Z}^d$ be finite, and assume $0 \in \ex(H(A))$ and $\spn(A) = \mathbb{R}^d$. Then from the structure theorem for convex polytopes \cite[Theorem 9.2]{Br83}, we know that there are linear maps $\beta_1,\dots,\beta_K, \gamma_1,\dots,\gamma_L, : \mathbb{R}^d \to \mathbb{R}$ for which
\[ H(A) = \bigcap\limits_{i=1}^K \{ x \in \mathbb{R}^d: \, \beta_i(x) \leqslant 1\} \cap \bigcap\limits_{j=1}^L \{x \in \mathbb{R}^d: \, \gamma_j(x) \geqslant 0\}\]
and the sets   $\{x \in H(A): \beta_i(x) = 1\}$ and $\{x \in H(A): \, \gamma_j(x) = 0\}$ form the facets of $H(A)$.  For each $i$ and $j$ we call $\{x \in H(A): \beta_i(x) = 1\}$ an \emph{outer facet} of $H(A)$ and $\{x \in H(A): \gamma_j(x) = 1\}$ an \emph{inner facet} of $H(A)$. 

We continue with a technical lemma which we will use to reduce $\wt(n)$ as discussed above. 

\begin{Lemma}[Preparation for reduction step]
\label{Lemma:reduction_prep}
Let $A \subset \mathbb{Z}^d$ be finite, and assume $0 \in \ex(H(A))$ and $\spn(A) = \mathbb{R}^d$. Let $S \subset A$, and suppose that $S$ does not lie in an outer facet of $H(A)$. Then for any linear map $\alpha : \mathbb{R}^d \to \mathbb{R}$ satisfying $\alpha(s) = 1$ for all $s \in S$ there exists $p \in H(A) \cap \spn(S) \cap \mathbb{Q}^d$ for which $\alpha(p) > 1$. 
\end{Lemma}

\begin{proof}
 Each outer facet of $H(A)$ is defined by $\{x \in H(A): \beta_i(x) = 1\}$ for some linear map $\beta_i$.
 Now $\beta_i(s) \leqslant 1$ for all $s \in S$ as $S \subset A \subset H(A)$, and we cannot have equality for all $s \in S $ as $S$ is not contained in any outer facet by the hypothesis, and so the \emph{barycentre}
\[
q := \frac{1}{\vert S \vert} \sum_{s \in S} s
\]
of $S$ satisfies $\beta_i(q) < 1$. Letting $\hat{\beta} = \max_i \beta_i(q) \in [0, 1)$, we see that $q$ lies inside $\hat{\beta} H(A)$, and so
for any $\eps \in (0, \hat{\beta}^{-1} - 1) \cap \mathbb{Q}$ the point $p = (1+\eps)q$ lies in $H(A)$.  This $p$ is clearly also in $\spn(S)$ and $\mathbb{Q}^d$, and satisfies $\alpha(p) = (1 + \eps)\alpha(q) = 1 + \eps >1$.
\end{proof}

We now use this observation to prove the existence of certain relations between sums of elements in $A$. 
\begin{Lemma}[Reduction step]
\label{Lemma:reduction_step}
Let $A \subset \mathbb{Z}^d$ be finite, and assume $0 \in \ex(H(A))$ and $\spn(A) = \mathbb{R}^d$. Suppose the elements of $S \subset A$ are linearly independent, and  that $S$ is not a subset of any outer facet of $H(A)$. Then there exist non-negative integers
 $\{\lambda_s\}_{s \in S}$ and $\{ \rho_a\}_{a \in A \setminus \{0\}}$ such that 
\[
 \sum_{s \in S} \lambda_s s = \sum_{a \in A \setminus \{0\}} \rho_a a \quad \text{and} \quad \wt(\lambda) > \wt(\rho),
\]
where $\lambda_s, \rho_a \leqslant \vol^{\dagger, \max}(H(A)) \text{ for all } s \in S\text{ and } a \in A\setminus  \{0\}$.
\end{Lemma}

\begin{proof} As $S$ is linearly independent, we know that $0\not\in S$ and there exists 
a linear map $\alpha: \mathbb{R}^d \to \mathbb{R}$  with $\alpha(s) = 1$ for all $s \in S$. 
From Lemma \ref{Lemma:reduction_prep}, choose $p \in H(A) \cap \spn(S) \cap \mathbb{Q}^d$ with $\alpha(p) > 1$. 
Therefore $p = \sum_{s \in S} \gamma_s s$ for some coefficients $\gamma_s \in \mathbb{Q}$, and 
$p= \sum_{a \in A} \delta_a a$ where $\wt(\delta) = 1$ and $\delta_a \in [0,1] \cap \mathbb{Q}$ for all $a \in A$.  Then
\[
\sum\limits_{s \in S} \gamma_s s = p = \sum\limits_{a \in A} \delta_a a \quad \text{and} \quad \wt(\gamma) = \alpha(p) > 1 = \wt(\delta).
\]
Let $L$ be the least common denominator of all the $\gamma_s$ and $\delta_a$.  Define $z_a = L(\delta_a - \gamma_a)$ for $a \in A\setminus\{0\}$, and
\[
z_0 = L\big(\wt(\gamma) - \wt(\delta) + \delta_0\big) > 0,
\]
so that 
\[
\wt(z)=z_0+\sum_{a \in A\setminus\{0\}} L(\delta_a - \gamma_a)=0.
\]
We then have $z\in  \mathcal{Z}$ (as defined in \eqref{eq:def_of_Z}, where we identify $\mathbb{Z}^A$ with $\mathbb{Z}^{\ell}$)
and $\supp(z^-) \subset S$.

 By Lemma \ref{Lemma rational combo}, we write $z = \sum_j \eta_j u_j$ with each $u_j \in \mathcal{Z}^{\dagger}$, $\eta_j \in \mathbb{Q}_{>0}$, and   $\supp(u_j^{\pm}) \subset \supp(z^{\pm})$. Now $0 \in \supp(z^+)$ as $z_0 >0$,  so $0\in \supp(u^+) \subset \supp(z^+)$ for some $u=u_j\in \mathcal{Z}^{\dagger}$ and $\supp(u^-) \subset \supp(z^-) \subset S$. Define
 \[
\lambda_s = \begin{cases} (-u_s) & \text{if } s \in \supp(u^-) \\
0 & \text{if } s \in S \setminus \supp(u^-), \end{cases}
\text{ and } \rho_a = \begin{cases} u_a & \text{if } a \in \supp(u^+) \setminus \{0\} \\
0 & \text{otherwise.} \end{cases}
\]
Then, since  $u \in \mathcal{Z}$, we have
\begin{align*}
\sum_{s\in S}  \lambda_s s = \sum_{\substack{a \in \supp(u^-)}}(-u_a)a &= \sum_{\substack{a \in \supp(u^+)}}u_a a = \sum_{a \in A \setminus \{0\}} \rho_a a \\
\text{and} \quad \wt(\lambda) = \sum_{\substack{a \in \supp(u^-)}}(-u_a) &=   \sum_{\substack{a \in \supp(u^+) }} u_a
> \wt(\rho).
\end{align*}
The final inequality uses the fact that $u_0>0$. The condition $ \max_{s\in S, a\in A}\lambda_s, \rho_a=\Vert u\Vert_{\infty} \leqslant \vol^{\dagger,\max}(H(A))$ then follows from Lemma \ref{Lemma circuit height bound}.
\end{proof}

Using the relation from the previous lemma, we can derive the decomposition $v = u+w$ as discussed at the start of the section. 

\begin{Lemma}[Regular representation]
\label{Lemma:regular representation}
Let $A \subset \mathbb{Z}^d$ be finite, and assume $0 \in \ex(H(A))$ and $\spn(A) = \mathbb{R}^d$. Let $v \in \mathcal{P}(A)$. Then there is a decomposition $v = u + w$ and an outer facet $F$ of $H(A)$ for which
\begin{itemize}
\item $w \in \mathcal{P}(B \cup \{0\})$, where $B = A \cap F$;
\item  $u \in MA$ where $M =(\vert A\vert - 1 - \vert B\vert) (\vol^{\dagger, \max}(H(A)) - 1)$.
\end{itemize}
\end{Lemma}

\begin{proof}
By definition we can write $v \in \mathcal{P}(A)$ as a $\mathbb{Z}_{ \geqslant 0}$-linear combination of the $a\in A$, so we select the representation
\[
v = \sum_{a \in A\setminus \{0\}} \eta_a a \text{ where each } \eta_a \in \mathbb{Z}_{ \geqslant 0}
\] 
for which $\wt(\eta)$ is minimal, and then for which 
\[
T = T((\eta_a)_a) := \{a \in A \setminus \{0\}: \, \eta_a \geqslant \vol^{\dagger, \max}(H(A))\}
\] 
 is also minimal. If $T$ is contained in an outer facet $F$ of $H(A)$ then we obtain the desired decomposition $v=u+w$, where 
\[
u := \sum_{a \in A \setminus (B \cup \{0\})} \eta_a a \in MA \quad \text{and} \quad w:=\sum_{b\in B: = A \cap F}  \eta_b b,
\] 
since $ \eta_a\leqslant \vol^{\dagger, \max}(H(A))-1$ for all $a \in A \setminus (B \cup \{0\})$, and $| A \setminus (B \cup \{0\})|=\vert A\vert - 1 - \vert B\vert$.
  
Henceforth we may assume that  no such facet $F$ exists, so that $T \neq \emptyset$. We obtain a contradiction as follows.

\textbf{Case I:} If the elements of $T$ are linearly independent then we apply Lemma \ref{Lemma:reduction_step} with $S:= T$ to obtain another representation of $v$,
\[
v = \sum_{a \in A\setminus \{0\}} \eta'_a a, \quad \text{where} \quad \eta'_a:= \begin{cases} \eta_a - \lambda_a+\rho_a &\text{if } a\in T;\\
\eta_a + \rho_a &\text{otherwise.}\end{cases}
\]
The coefficients $\eta_a'$ are all non-negative since each $\eta_a , \rho_a\geqslant 0$ and 
\[
\lambda_t \leqslant \vol^{\dagger,\max}(H(A)) \leqslant \eta_t \text{ for all } t \in T. 
\]
However
\[
\wt(\eta') = \wt(\eta) + \wt(\rho) - \wt(\lambda) < \wt(\eta),
\]
contradicting the minimality of $\wt(\eta)$. 
  
\textbf{Case II:} Otherwise the elements of $T$ are linearly dependent and so there exist $z_t \in \mathbb{Z}$, not all zero, for which
  \[ \sum\limits_{t \in T} z_t t = 0. \]  
Define $z_0:=- \sum_{t \in T} z_t$, and multiply through all the $z_v$-values by $-1$ if necessary to ensure that $z_0\geqslant 0$.
As usual, we define $z_a=0$ for all $a\in A$ on which it is not yet defined, so we can consider $z$ as a non-zero element of $\mathbb{Z}^A$ with $z \in \mathcal{Z}$. By Lemma \ref{Lemma circuits exist} there then exists $\mu \in \mathcal{Z}^{\dagger}$ with $\supp(\mu^{\pm}) \subset \supp(z^{\pm})$, and by Lemma \ref{Lemma circuit height bound} we have 
$\lVert \mu\rVert_{\infty} \leqslant \vol^{\dagger,\max}(H(A))$.
  
\textbf{Case IIa:} If $\mu_0 \neq 0$ then we must have $\mu_0 > 0$ since $\supp(\mu^\pm) \subset \supp(z^\pm)$ and $z_0 > 0$.  Now write $v = \sum_{a \in A} \eta'_a a$, where $\eta'_a=\eta_a+\mu_a$ for $a \neq 0$ and $\eta'_0 = 0$.  We have $\eta_a' \geqslant 0$ for all $a$, since $\eta_a'$ agrees with $\eta_a \geqslant 0$ unless $a \in T$, in which case $\eta_a \geqslant \vol^{\dagger, \max}(H(A))$ and $\mu_a \geqslant -\lVert \mu \rVert_\infty \geqslant -\vol^{\dagger, \max}(H(A))$.  But we also have
\[
\wt(\eta') = \wt(\eta) + \wt(\mu) - \mu_0 = \wt(\eta) - \mu_0 < \wt(\eta),
\]
contradicting the minimality of $\wt(\eta)$. 
 
\textbf{Case IIb:} Otherwise $\mu_0 = 0$. Then pick $n \in \mathbb{N}$ maximal such that $\eta' := \eta - n \mu$ has all components non-negative.  We obtain $v = \sum_{a \in A} \eta'_a a$ and $\wt(\eta') = \wt(\eta)$.  But we must have $\eta'_t < \vol^{\dagger,\max}(H(A))$ for some $t\in T$, otherwise we can increase $n$, so $ T(\eta')$ must be a proper subset of $T$, contradicting minimality of $T$.
\end{proof}

In order to leverage the decomposition $v= u+w$ to show that $v \in NA$, we need to control how negative the evaluation $\beta(u)$ can get, when $\beta$ defines an outer facet of $H(A)$. This is the purpose of the next lemma.  Recall from Definition \ref{def: kappa} and Remark \ref{rmk: kappa} that
\[
\kappa(A) = \max_F \frac{\max_a g_F(a)}{\min_a g_F(a)},
\]
where $F$ ranges over facets of $H(A)$, $a$ ranges over points of $A \setminus F$, and $g_F$ is any affine-linear function $\mathbb{R}^d \to \mathbb{R}$ which vanishes on $F$ and is strictly positive on $H(A) \setminus F$.

\begin{Lemma}[Negative coefficients]
\label{Lemma: negative coefficients}
Let $A \subset \mathbb{R}^d$ be finite with $0 \in \ex(H(A))$ and $\spn(A) = \mathbb{R}^d$. Let $\beta: \mathbb{R}^d \to \mathbb{R}$ be a linear map for which $F = \{x \in H(A): \, \beta(x) =1\}$ is an outer facet of $H(A)$. Then, for all $a \in A$,\[\beta(a) \geqslant 1 - \kappa(A).\]
\end{Lemma}
\begin{proof}
The facet $F$ of the $d$-dimensional convex polytope $H(A)$ is the convex hull of at least $d$ points of $A$ (see \cite[Lemma A.2 (4)]{GSW23} for a discussion). In particular there are $d$ linearly independent $b^{(1)},\dots,b^{(d)} \in A$ for which $\beta(b^{(j)}) = 1$ (and these uniquely determine $\beta$). Let $b^{(j)}_i$ denote the $i^{th}$ coordinate of $b^{(j)}$ with respect to the standard basis, and for $a \in A$ let $a_i$ denote the $i^{th}$ coordinate with respect to the standard basis. Expressing $\beta$ in coordinates and computing the necessary matrix inverses, we derive \[ \beta(a) = \frac{1}{\det B_{\mat}} \sum_{i,j \leqslant d} (-1)^{i+j}a_i M_{ij},\] where $B_{\mat}$ is the $d$-by-$d$ matrix with $(B_{\mat})_{ij} = b^{(j)}_i$, and $M_{ij}$ is the minor formed by deleting the $i^{th}$ row and $j^{th}$ column of $B_{\mat}$  and taking the determinant. Yet \[\det B_{\mat} - \sum_{i,j \leqslant d} (-1)^{i+j}a_i M_{ij} = \det \left( \begin{matrix}
1 & 1 & \cdots & 1 \\
a_1 & b^{(1)}_1 & \cdots & b^{(d)}_1 \\
\vdots & \vdots & & \vdots \\
a_d & b^{(1)}_d & \cdots & b^{(d)}_d \\
\end{matrix}\right),\] as can be seen from expanding the determinant along the top row, and from column operations we have \[\det \left( \begin{matrix} 1 & 1 & \cdots & 1 \\
a_1 & b^{(1)}_1 & \cdots & b^{(d)}_1 \\
\vdots & \vdots & & \vdots \\
a_d & b^{(1)}_d & \cdots & b^{(d)}_d
\end{matrix}\right) = \det(b^{(1)} - a, \cdots, b^{(d)} - a).\]
Letting $g_F(a) = \det(b^{(1)} - a, \cdots, b^{(d)} - a)$, and assuming that the $b^{(j)}$ are ordered so that $\det B_{\mat} = g_F(0)$ is positive, we obtain
\[ \beta(a) = 1 - \frac{\det(b^{(1)} - a, \cdots, b^{(d)} - a)}{\det B_{\mat}} = 1 - \frac{g_F(a)}{g_F(0)} \geqslant 1 - \kappa(A),\] as required. 
\end{proof}

\begin{Remark}
\emph{Less explicitly, one can argue that $1-\beta$ and $\frac{g_F}{g_F(0)}$ are the unique affine-linear functions $\mathbb{R}^d \to \mathbb{R}$ which vanish on $F$ and map $0$ to $1$, so they must agree.}
\end{Remark}

We can now deduce part of Theorem \ref{Theorem:improved_structural_threshold}:

\begin{proof}[Proof of bound \eqref{eq:second_structure_bound}]
Let \[v \in (NH(A) \cap \mathbb{Z}^d) \setminus \Big( \bigcup\limits_{b \in \ex(H(A))} (bN - \mathcal{E}(b-A))\Big)\] 
where $N\geqslant (d+1)N_0$ with 
$ N_0: = \kappa(A)(\vert A\vert - d - 1) \vol^{\dag,\max}(H(A))$.

Since $v \in NH(A)$ there is some subset $S= \{s_0,\dots,s_{d}\} \subset \ex(H(A))$ with $v \in NH(S)$, by Caratheodory's theorem \cite[Corollary 2.5]{Br83}. Writing $v = \sum_{s \in S} c_s s$, with $c_s \geqslant 0$ for all $s$ and $\wt(c) = N$, there must be some $c_s \geqslant N_0$ as $N \geqslant (d+1)N_0$. By re-labelling the vectors in $S$ we may assume that $c_{s_0} \geqslant N_0$, and then
\[
 v^\prime: = s_0N - v = \sum_{s \in S \setminus \{s_0\}} c_s(s_0 - s) \in (N - c_{s_0})H(s_0 - S) \subset (N - N_0) H(s_0 - S).
 \]
Letting $A^\prime = s_0 - A$, we have $s_0 - S \subset A^\prime$, so by the preceding equation $v^\prime$ is contained in $(N - N_0)H(A^\prime)$.  We also have, by assumption, that $v \notin s_0 N - \mathcal{E}(s_0 - A)$, so $v' \notin \mathcal{E}(s_0 - A) = \mathcal{E}(A')$. We conclude that $v' \in \mathcal{P}(A')$. Note also that $\vert A^\prime\vert = \vert A\vert$, $\vert \ex(H(A^\prime)) \vert = \vert \ex(H(A))\vert$, $\vol(H(A^\prime)) = \vol(H(A))$, $\vol^{\dag,\max}(A^\prime) = \vol^{\dag,\max}(A)$, and the same for $\min$.  

Now apply Lemma \ref{Lemma:regular representation} to $v^\prime$, with $A^\prime$ in place of $A$.  We see that there exists an outer facet $F$ of $H(A^\prime)$ such that we can write $v'=u+w$, where
\[
w=\sum_{b\in B} \lambda_b b \quad \text{and} \quad u=\sum\limits_{a \in A^\prime} \eta_a a
\]
with $B: = A^\prime \cap F$, all $\eta_a,\lambda_b \in \mathbb{Z}_{ \geqslant 0}$, and 
\begin{equation} \label{eq: step on the way}
\wt(\eta) \leqslant (\vert A^\prime\vert  - \vert B\vert-1)(\vol^{\dag,\max}(H(A^\prime)) - 1) \leqslant  (\vert A\vert - d-1 ) \vol^{\dag,\max}(H(A)).
\end{equation} 
The second inequality here uses that $\vert B\vert \geqslant d$, which follows from the fact that every facet of the $d$-dimensional convex polytope $H(A^\prime)$ is the convex hull of at least $d$ points of $A^\prime$.

We know that $F=\{x \in H(A^\prime): \, \beta(x) = 1\}$ for some linear map $\beta:\mathbb{R}^d\to \mathbb{R}$. 
As $v^\prime \in (N-N_0)H(A^\prime)$ we have
\[
N-N_0 \geqslant \beta(v) = \beta(u)+\beta(w) = \sum\limits_{a \in A^\prime} \eta_a \beta(a) + \wt(\lambda)
\]
as $\beta(b)=1$ for each $b\in B$. Moreover, combining \eqref{eq: step on the way} with Lemma \ref{Lemma: negative coefficients} applied to $A^\prime$ gives
\[
\wt(\eta) - \sum_{a \in A^\prime} \eta_a \beta(a) = \sum\limits_{a \in A^\prime} \eta_a (1-\beta(a))\leqslant  \kappa(A) \wt(\eta) \leqslant N_0.
\]
Summing the last two inequalities we then obtain
\[ 
\wt(\eta) + \wt(\lambda) \leqslant N
\]
and so $v^\prime \in NA^\prime$. Therefore $v=s_0N -v^\prime\in s_0N - NA^\prime=N(s_0-A^\prime)=NA$ as required. 
\end{proof}

It remains to prove the bound \eqref{eq:first_structure_bound}, which separates the contribution from $\ex(H(A))$. To effect this separation, we begin with an argument about triangulating polytopes. 

\begin{Lemma}[Splitting $A$ into simplices centred at the origin]
\label{Lemma:simplexdecomp}
Let $A \subset \mathbb{Z}^d$ be finite with $0 \in \ex(H(A))$ and $\operatorname{span}(A) = \mathbb{R}^d$. Then $H(A)$ may be partitioned as a finite union of simplices $\cup_{j} H(B^{(j)} \cup \{0\})$, where each $B^{(j)} \subset \ex(H(A))$ is a basis of $\mathbb{R}^d$, and for each $i \neq j$ the set $H(B^{(i)} \cup \{0\}) \cap H(B^{(j)} \cup \{0\})$ is contained in a subspace of dimension at most $d-1$. In particular, $H(B^{(i)} \cup \{0\}) \cap H(B^{(j)} \cup \{0\})$ has zero measure. 
\end{Lemma}
\noindent When $d=2$, this is the obvious statement that any polygon with a vertex at the origin may be decomposed into disjoint triangles, all of which have a common vertex at the origin. 
\begin{proof}
The $d=1$ case is trivial, so assume that $d \geqslant 2$. We will induct on dimension. Let $F_1,\dots,F_K$ denote the list of outer facets of $H(A)$. Each $F_i$ is a convex polytope of dimension $d-1$, generated by points in $\ex(H(A))$. Therefore, by the induction hypotheses, one may decompose $F_i$ as a union of $(d-1)$-dimensional simplices of the form $H(\{a_1,\dots,a_d\})$, where $\{a_1,\dots,a_d\} \subset \ex(H(A))$ is linearly independent and the intersection of any two of these simplices is contained in an affine subspace of dimension at most $d-2$. (In fact one may further assume that there is a common vertex $a_1$ to all these simplices, but that will not be necessary for the induction step.) 

Choose $B^{(j)}$ to be the list of such sets $\{a_1,\dots,a_d\}$, taken over all the facets $F_1,\dots,F_K$. We claim that these $B^{(j)}$ satisfy the requirements of the lemma. By construction, each $B^{(j)} \subset \ex(H(A))$ is a basis of $\mathbb{R}^d$. To show that the union of the $H(B^{(j)} \cup \{0\})$ is $H(A)$, fix $x \in H(A)\setminus\{0\}$ and pick (the unique) $\lambda_x \geqslant 1$ such that $\lambda_x x \in \cup_K F_K$.  Then $\lambda_x x \in H(B^{(j)})$ for some $B^{(j)}$. Thus there exist coefficients $c_b$ for $b \in B^{(j)}$ such that $c_b \geqslant 0$, $\wt(c) = 1$, and
\[ \lambda_x  x = \sum_{b \in B^{(j)}} c_b b.\] 
We therefore have
\[ 
x = \Big(1 - \frac{1}{\lambda_x}\Big) 0 +  \sum\limits_{b \in B^{(j)}} \frac{c_b}{\lambda_x} b  \in H(B^{(j)} \cup \{0\}),
\]
as wanted.

It remains to show that each intersection $H(B^{(i)} \cup \{0\}) \cap H(B^{(j)} \cup \{0\})$ is contained in a subspace of dimension at most $d-1$.  So fix an arbitrary non-zero $x \in H(B^{(1)} \cup \{0\}) \cap H(B^{(2)} \cup \{0\})$.  There are coefficients $c_i^{(1)}, c_i^{(2)} \geqslant 0$ with
\[ 
x = \sum_{i \leqslant d} c_i^{(1)} b^{(1)}_i = \sum_{i \leqslant d} c_i^{(2)} b^{(2)}_i
\] 
and  $0< \wt(c^{(1)}), \wt(c^{(2)}) \leqslant 1$. Letting $\wt_j$ denote $\wt(c^{(j)})$, and assuming WLOG that $\wt_2 \geqslant \wt_1$, we can re-scale to obtain
 \[
 y := \frac{x}{\wt_2} = \sum_{i \leqslant d} \frac{c_i^{(1)}}{\wt_2} b_i^{(1)} = \sum_{i \leqslant d} \frac{c_i^{(2)}}{\wt_2} b_i^{(2)},
 \] 
which lies in $H(B^{(1)} \cup \{0\})\cap H(B^{(2)})$ since $\wt(\frac{c^{(1)}}{{\wt_2}}) \leqslant 1$ and $\wt(\frac{c^{(2)}}{{\wt_2}}) = 1$.
 
Suppose for contradiction that $\wt(\frac{c^{(1)}}{{\wt_2}}) <  1$, so $(1+\eps)y\in H(B^{(1)} \cup \{0\}) \subset H(A)$ for all sufficiently small $\varepsilon>0$. Let $B^{(2)}$ be a subset of the outer facet defined by the linear map $\beta^{(2)}: \mathbb{R}^d \to \mathbb{R}$, so that 
 $B^{(2)} \subset \{ u \in \mathbb{R}^d: \, \beta^{(2)}(u) = 1\}$ and  $H(A) \subset \{u \in \mathbb{R}^d: \, \beta^{(2)}(u) \leqslant 1\}$. 
 Then for all sufficiently small $\varepsilon > 0$ we have
  \[ 1\geqslant \beta^{(2)}((1+\eps)y) = (1+\eps) \beta^{(2)}(y) = 1+\eps  > 1.\]
This gives the desired contradiction, and we deduce that  $\wt(\frac{c^{(1)}}{{\wt_2}}) = 1$.  So 
\[ y \in H(B^{(1)}) \cap H(B^{(2)}),\]
and because $x=\wt_2 y$ for some $\wt_2\in [0,1]$ we conclude that 
 \[ 
 H(B^{(1)} \cup \{0\}) \cap H(B^{(2)} \cup \{0\}) = H((H(B^{(1)}) \cap H(B^{(2)})) \cup \{0\}) .
 \] 
 Hence $H(B^{(1)} \cup \{0\}) \cap H(B^{(2)} \cup \{0\})$  is contained in a subspace of dimension at most $d-1$  by the induction hypothesis.
\end{proof}
Using this decomposition, we can generalise an additive combinatorial argument from \cite{GSW23} and \cite{CG21} (which was applied when $H(A)$ was a $d$-simplex).

\begin{Lemma}[Restricting the influence of non-extremal elements]
\label{Lemma:minelement_decomp}
Let $A \subset \mathbb{Z}^d$ be a finite set with $0 \in \ex(H(A))$ and $\operatorname{span}(A) = \mathbb{R}^d$. Then there exists a finite set $S = d! \vol(H(A)) A$ for which \[\mathcal{P}(A) = S + \mathcal{P}(\ex(H(A))).\]
\end{Lemma}

The proof is similar to (but simpler than) \cite[Lemma 3.2]{GSW23} with the set $B := \ex(H(A))$. 

\begin{proof}
Let $v \in NA$. We will show that $v \in S + \mathcal{P}(\ex(H(A)))$ by induction on $N$. For $N \leqslant d! \vol(H(A))$ we have $v \in NA \subset S \subset S + \mathcal{P}(\ex(H(A)))$. 

Suppose that $N > d! \vol(H(A))$. We can write $v = a_1 + a_2 + \cdots + a_{N}$ with each $a_i \in A$. By Lemma \ref{Lemma:simplexdecomp}, there is a  partition $H(A) = \cup_{j} H(B^{(j)} \cup \{0\})$ where each $B^{(j)} \subset \ex(H(A))$. Therefore we can partition $\{1,\dots,N\} = \cup_j T_j$ to obtain
\[ v = \sum_{j} \sum_{i \in T_j} a_i,\] 
where $i\in T_j$ implies that $a_i \in H(B^{(j)} \cup \{0\})$.

Since $\vol(H(A)) = \sum_{j} \vol(H(B^{(j)} \cup \{0\}))$ by Lemma \ref{Lemma:simplexdecomp} there is some $j$ for which 
\[\vert T_j\vert > d! \vol(H(B^{(j)} \cup \{0\}))= \vert \mathbb{Z}^d/\Lambda_{B^{(j)} \cup \{0\} }\vert ,\] 
by the pigeonhole principle.  Reordering the indices on the $a_i$ we   write $T_j = \{1,\dots,\vert T_j\vert\}$. Two of the $\vert T_j\vert$ partial sums 
\[ a_1,a_1 + a_2,\dots,a_1 + a_2 + \cdots + a_{\vert T_j\vert} \text{ mod } \Lambda_{B^{(j)} \cup \{0\} },\] 
must be congruent to each other mod $\Lambda_{B^{(j)} \cup \{0\} }$
by the pigeonhole principle. Their difference yields a non-trivial partial sum $\sum_{i \in I} a_i \equiv 0 \text{ mod } \Lambda_{B^{(j)} \cup \{0\} }$ (where $I \subset T_j$ is a non-empty interval) and so this partial sum can be replaced by a sum of elements from $B^{(j)} \cup \{0\}$. Therefore
 \[
 \sum_{i \in I} a_i \in \mathcal{P}(B^{(j)} \cup \{0\}) \subset \mathcal{P}(\ex(H(A))).
 \] 
 By the induction hypothesis, we have $v - \sum_{i \in I} a_i \in S+ \mathcal{P}(\ex(H(A)))$, and so 
\[v \in S + \mathcal{P}(\ex(H(A))) + \mathcal{P}(\ex(H(A))) \subset S + \mathcal{P}(\ex(H(A)))\] as required. 
\end{proof}

We are now ready to finish the argument by modifying the proof of \eqref{eq:second_structure_bound}.

\begin{proof}[Proof of bound \eqref{eq:first_structure_bound}]
Let
\[
v \in (NH(A) \cap \mathbb{Z}^d) \setminus \Big( \bigcup\limits_{b \in \ex(H(A))} (bN - \mathcal{E}(b-A))\Big),
\]
where $N\geqslant (d+1)N_0$ with
\[
 N_0: = \kappa(A) \Big(d!\vol(H(A)) + (\vert\ex(H(A))\vert - d - 1) \vol^{\dag,\max}(H(A))\Big).
\]

As in the proof of \eqref{eq:second_structure_bound} we use  Caratheodory's theorem to determine some $s_0\in \ex(H(A))$ for which 
 \[ 
 v^\prime: = s_0N - v   \in (N - N_0)H(A^\prime)\cap \mathcal{P}(A^\prime),
 \] 
where $A^\prime:= s_0 - A$. By Lemma \ref{Lemma:minelement_decomp} applied to $A^\prime$, we may write $v^\prime = y + x $ where $y \in d! \vol(H(A^\prime)) A^\prime$ and $x \in \mathcal{P}(\ex(H(A^\prime))) $. Applying Lemma \ref{Lemma:regular representation} to $x \in \mathcal{P}(\ex(H(A^\prime)))$ (in place of $v\in \mathcal{P}(A)$) we write $x = u + w$, where $w \in \mathcal{P}(B \cup \{0\})$ and $u \in M \ex(H(A'))$, with $B = \ex(H(A')) \cap F$ for some outer facet $F$ of $H(A')$, and $M =(\vert \ex(H(A'))\vert - 1 - \vert B\vert) (\vol^{\dagger, \max}(H(A')) - 1)$.

Now let
\[
z = y+u=\sum\limits_{a \in A^\prime} \rho_aa \quad \text{and} \quad w = \sum_{b \in B} \lambda_b b,
\]
and note that $\rho_a, \lambda_b \in \mathbb Z_{\geqslant 0}$ for all $a$ and $b$.  We obtain $v=z+w$ and
 \begin{align}
\label{eq:bound_on_lambdasum_2}
\wt(\rho)  \leqslant d! \vol(H(A)) + (\vert \ex(H(A))\vert - 1 - d)\vol^{\dag,\max}(H(A)),
\end{align} 
using $\vert \ex(H(A^\prime))\vert=\vert \ex(H(A))\vert$, $\vol^{\dag,\max}(H(A^\prime))=\vol^{\dag,\max}(H(A))$, and $\vert B\vert \geqslant d$, as in the proof of \eqref{eq:second_structure_bound}.  The outer facet $F$ is given by $\{x \in H(\ex(H(A^\prime))): \, \beta(x) = 1\}$ for some linear $\beta:\mathbb{R}^d\to \mathbb{R}$ so we again obtain
\[
N-N_0 \geqslant \beta(v) = \beta(z)+\beta(w) = \sum\limits_{a \in A^\prime} \rho_a \beta(a) + \wt(\lambda).
\]
By again applying Lemma \ref{Lemma: negative coefficients} to $A^\prime$, this time using the bound \eqref{eq:bound_on_lambdasum_2} in place of \eqref{eq: step on the way}, we obtain
$\wt(\rho) + \wt(\lambda)  \leqslant N$
with the modified value for $N_0$,
and so $v^\prime \in NA^\prime$. Therefore $v=s_0N -v^\prime\in s_0N - NA^\prime=N(s_0-A^\prime)=NA$ as required. 
\end{proof}

\bibliographystyle{plain}
\bibliography{bib}

\begin{thebibliography}{10}

\bibitem{BR15}
M.~Beck and S.~Robins.
\newblock {\em Computing the continuous discretely}.
\newblock Undergraduate Texts in Mathematics. Springer, New York, second
  edition, 2015.
\newblock Integer-point enumeration in polyhedra, With illustrations by David
  Austin.

\bibitem{BV83}
E.~Bombieri and J.~Vaaler.
\newblock On {S}iegel's lemma.
\newblock {\em Invent. Math.}, 73(1):11--32, 1983.

\bibitem{Br83}
A.~Br\o~ndsted.
\newblock {\em An introduction to convex polytopes}, volume~90 of {\em Graduate
  Texts in Mathematics}.
\newblock Springer-Verlag, New York-Berlin, 1983.

\bibitem{CG21}
M.~J. Curran and L.~Goldmakher.
\newblock Khovanskii's theorem and effective results on sumset structure.
\newblock {\em Discrete Anal.}, pages Paper No. 27, 25, 2021.

\bibitem{E62}
E.~Ehrhart.
\newblock Sur les poly\`edres homoth\'{e}tiques bord\'{e}s \`a {$n$}
  dimensions.
\newblock {\em C. R. Acad. Sci. Paris}, 254:988--990, 1962.

\bibitem{E95}
D.~Eisenbud.
\newblock {\em Commutative algebra}, volume 150 of {\em Graduate Texts in
  Mathematics}.
\newblock Springer-Verlag, New York, 1995.
\newblock With a view toward algebraic geometry.

\bibitem{GS20}
A.~Granville and G.~Shakan.
\newblock The {F}robenius postage stamp problem, and beyond.
\newblock {\em Acta Math. Hungar.}, 161(2):700--718, 2020.

\bibitem{GSW23}
A.~Granville, G.~Shakan, and A.~Walker.
\newblock Effective results on the size and structure of sumsets.
\newblock {\em Combinatorica}, 43(6):1139--1178, 2023.

\bibitem{GW21}
A.~Granville and A.~Walker.
\newblock A tight structure theorem for sumsets.
\newblock {\em Proc. Amer. Math. Soc.}, 149(10):4073--4082, 2021.

\bibitem{K92}
A.~G. Khovanski\u{\i}.
\newblock The {N}ewton polytope, the {H}ilbert polynomial and sums of finite
  sets.
\newblock {\em Funktsional. Anal. i Prilozhen.}, 26(4):57--63, 96, 1992.

\bibitem{L22}
V.~F. Lev.
\newblock The structure of higher sumsets.
\newblock {\em Proc. Amer. Math. Soc.}, 150(12):5165--5177, 2022.

\bibitem{MP18}
J.~McCullough and I.~Peeva.
\newblock Counterexamples to the {E}isenbud-{G}oto regularity conjecture.
\newblock {\em J. Amer. Math. Soc.}, 31(2):473--496, 2018.

\bibitem{N72}
M.~B. Nathanson.
\newblock Sums of finite sets of integers.
\newblock {\em Amer. Math. Monthly}, 79:1010--1012, 1972.

\bibitem{NR02}
M.~B. Nathanson and I.~Z. Ruzsa.
\newblock Polynomial growth of sumsets in abelian semigroups.
\newblock {\em J. Th\'{e}or. Nombres Bordeaux}, 14(2):553--560, 2002.

\bibitem{S96b}
B.~Sturmfels.
\newblock Equations defining toric varieties.
\newblock arXiv:alg-geom/9610018.

\bibitem{S96}
B.~Sturmfels.
\newblock {\em Gr\"{o}bner bases and convex polytopes}, volume~8 of {\em
  University Lecture Series}.
\newblock American Mathematical Society, Providence, RI, 1996.

\bibitem{S96a}
B.~Sturmfels.
\newblock Equations defining toric varieties.
\newblock In {\em Algebraic geometry---{S}anta {C}ruz 1995}, volume 62, Part 2
  of {\em Proc. Sympos. Pure Math.}, pages 437--449. Amer. Math. Soc.,
  Providence, RI, 1997.

\bibitem{WCC11}
J.-D. Wu, F.-J. Chen, and Y.-G. Chen.
\newblock On the structure of the sumsets.
\newblock {\em Discrete Math.}, 311(6):408--412, 2011.

\end{thebibliography}

\end{document}